\documentclass{amsart}
\usepackage{graphicx}
\usepackage{amsmath}
\usepackage{amssymb}
\usepackage[all]{xy}

\newcommand{\Hom}{\operatorname{Hom}\nolimits}
\newcommand{\D}{\operatorname{D}\nolimits}
\renewcommand{\Im}{\operatorname{Im}\nolimits}
\renewcommand{\mod}{\operatorname{mod}\nolimits}
\newcommand{\stmod}{\operatorname{\underline{mod}}\nolimits}
\newcommand{\stHom}{\operatorname{\underline{Hom}}\nolimits}
\newcommand{\Ker}{\operatorname{Ker}\nolimits}

\newcommand{\Tor}{\operatorname{Tor}\nolimits}
\newcommand{\Ext}{\operatorname{Ext}\nolimits}
\newcommand{\Tateext}{\operatorname{\widehat{Ext}}\nolimits}

\newcommand{\HH}{\operatorname{HH}\nolimits}

\renewcommand{\H}{\operatorname{H}\nolimits}

\newcommand{\m}{\operatorname{\mathfrak{m}}\nolimits}

\newcommand{\La}{\Lambda}

\newcommand{\op}{\operatorname{op}\nolimits}

\newcommand{\T}{\operatorname{\mathcal{T}}\nolimits}

\newcommand{\e}{\operatorname{e}\nolimits}
\newcommand{\Lae}{\Lambda^{\e}}

\newcommand{\s}{\operatorname{\Sigma}\nolimits}

\newcommand{\TateHH}{\operatorname{\widehat{HH}}\nolimits}

\newtheorem{theorem}{Theorem}[section]

\newtheorem{corollary}[theorem]{Corollary}

\newtheorem{lemma}[theorem]{Lemma}
\newtheorem{proposition}[theorem]{Proposition}
\theoremstyle{definition}
\newtheorem*{definition}{Definition}
\theoremstyle{definition}

\theoremstyle{definition}

\theoremstyle{definition}
\newtheorem*{example}{Example}
\theoremstyle{definition}

\theoremstyle{definition}

\theoremstyle{remark}
\newtheorem*{remark}{Remark}
\theoremstyle{definition}

\theoremstyle{definition}

\begin{document}

\title{The Negative Side of Cohomology for Calabi-Yau Categories}
\author{Petter Andreas Bergh, David A.\ Jorgensen \& Steffen Oppermann}

\address{Petter Andreas Bergh \\ Institutt for matematiske fag \\
  NTNU \\ N-7491 Trondheim \\ Norway}
\email{bergh@math.ntnu.no}

\address{David A.\ Jorgensen \\ Department of mathematics \\ University
of Texas at Arlington \\ Arlington \\ TX 76019 \\ USA}
\email{djorgens@uta.edu}

\address{Steffen\ Oppermann \\ Institutt for matematiske fag \\
  NTNU \\ N-7491 Trondheim \\ Norway}
\email{steffen.oppermann@math.ntnu.no}

\date{\today}

\subjclass[2010]{13H10, 16E40, 16W50}

\keywords{Graded algebras, negative products, Calabi-Yau categories, }

\begin{abstract}
We study $\mathbb Z$-graded cohomology rings defined over Calabi-Yau categories. We show that the cohomology in negative degree is a trivial extension of the cohomology ring in non-negative degree, provided the latter admits a regular sequence of central elements of length two. In particular, the product of elements of negative degrees are zero. As corollaries we apply this to Tate-Hochschild cohomology rings of symmetric algebras, and to Tate cohomology rings over group algebras.  We also prove similar results for Tate cohomology rings over commutative local Gorenstein rings.
\end{abstract}

\maketitle

\section{Introduction}\label{intro}

Let $\T$ be an $n$-Calabi-Yau triangulated category, and $X$ an object in $\T$.  We investigate the structure of the
$\mathbb Z$-graded cohomology ring
\[
\Hom_{\T}^*(X,X)=\oplus_{i\in\mathbb Z}\Hom_{\T}(X,\s^iX).
\]
One of our main results states that when the subalgebra $\oplus_{i\ge 0}\Hom_{\T}(X,\s^iX)$ possesses a regular sequence of length 2 of central elements of positive degree, then $n\le -1$ and most products in negative degree are trivial. More precisely, if $x,y\in \Hom_{\T}^*(X,X)$ with $|x|\le n$ and $|y|<0$, then $xy=0$.  In particular, when $n=-1$, which is a common case in applications, then one can remove the word ``most" from the previous sentence, and in fact, the cohomology in negative degree is a trivial extension of the cohomology ring in non-negative degree.

The main situation to which our results apply is when $\T$ is the stable module category of a finite dimensional
symmetric algebra $\Lambda$ over a field.  Here our results translate to the Tate cohomology ring $\Tateext^*_{\La}(M,M)$ of a finitely generated
(left) $\La$-module $M$ having trivial products in negative degree.  As a special case we recover a result of
Benson and Carlson \cite{BensonCarlson} on the Tate cohomology ring $\Tateext^*_{kG}(k,k)$, for $G$ a finite group whose order is divisible by the characteristic of the ground field $k$. Indeed, their paper served as a major inspiration for the present one.
Another corollary is that the Tate-Hochschild cohomology ring $\TateHH^*( \La )$ of a symmetric algebra $\Lambda$, defined in \cite{BerghJorgensen}, has trivial products in negative degree.
We also give various results on the structure of the cohomology rings when
$\oplus_{i\ge 0}\Hom_{\T}(X,\s^iX)$ is only assumed to admit a single central regular element of positive degree.

In Section \ref{A} we derive properties in the abstract for products in $\mathbb Z$-graded associative algebras which possess a natural duality.  The key point here is that when the nonnegative side of such algebras admit a regular sequence of length 1 or 2 contained in the positive side of the algebra, then products in negative degrees are often zero.
These results lay the foundation for the rest of the paper.

In Section \ref{B} we apply the results of Section \ref{A} to cohomology rings of objects in $n$-Calabi-Yau categories.  This includes the specific results on Tate and Tate-Hochschild cohomology mentioned above.

Finally, in Section \ref{C} we prove similar results for Tate cohomology rings over commutative Noetherian local rings, in particular Gorenstein rings. For example, we show that when $(R,\m)$ is a zero-dimensional local Gorenstein ring, $M$ is a finitely generated $R$-module, and $\oplus_{i\ge 0}\Ext_{R}^i(M,M)$ contains a regular sequence of length two, then in $\Tateext_{R}^*(M,M)$ the products of negative degree are zero (cf. \cite{AvramovVeliche} where the structure of $\Tateext^*_R(R/\m,R/\m)$ is treated in any dimension.)

\section{Products in $\mathbb Z$-graded algebras}\label{A}

Throughout this section we let $A = \bigoplus_{i \in \mathbb{Z}} A^i$ be a $\mathbb Z$-graded ring, that is, a unital ring graded by the integers.  For an integer $n$ we write $A^{\ge n} = \bigoplus_{i \geq n} A^i$ (and similarly for $A^{> n}$, $A^{\leq n}$, \ldots). We will be particularly interested in the non-negative subalgebra $A^{\ge 0}=\oplus_{i\ge 0} A^i$ of $A$.
For any graded $A$-module $M$ and any integer $t$, the shifted graded module $M[t]$ is defined by $M[t]^i = M^{i+t}$.

\begin{definition}
The $\mathbb Z$-graded ring $A$ has \emph{non-degenerate products in degree $n$} if for any integer $i$ and any $a \in A^i \setminus \{ 0 \}$ there are $b, c \in A^{n-i}$ such that
\[ ab \neq 0 \neq ca. \]
\end{definition}

If $A$ is also a $k$-algebra, for $k$ a field (always assumed to be concentrated in degree zero) we let $\D A$ denote the graded dual of $A$, that is the bimodule given by $(\D A)^i = \Hom_k(A^{-i}, k)$.

\begin{definition}
Assume $A$ is a $\mathbb Z$-graded $k$-algebra over a field $k$. Then $A$ is \emph{$n$-shifted selfdual} if
\[ A[n] \cong \D A \]
as left (or equivalently as right) $A$-modules.
\end{definition}

\subsection*{Bilinear forms}

For a graded ring $A$, consider a graded bilinear form
\[ \langle -, - \rangle \colon A \times A \to X, \]
where $X$ is some abelian group, concentrated in degree $n$. (If $A$ additionally is a $k$-algebra, we assume $X$ to be a $k$-module and the form to be $k$-bilinear.)
This means we have a family of bilinear maps
\[ \langle -, - \rangle^i \colon A^i \times A^{n-i} \to X. \]

\begin{definition}
The graded bilinear form $\langle -, - \rangle:A\times A \to X$ is called
\begin{itemize}
\item \emph{non-degenerate} if for $i\in\mathbb Z$ and $a \in A^i \setminus \{0\}$ there are
$b, c \in A^{n-i}$ such that $\langle a, b \rangle \neq 0 \neq \langle c, a \rangle$.
\item \emph{associative} if for any $a \in A^i$, $b \in A^j$, and $c \in A^{n-i-j}$ we have $\langle ab, c \rangle = \langle a, bc \rangle$.
\end{itemize}
\end{definition}

\begin{proposition}
Let $A$ be a $\mathbb Z$-graded ring. Then $A$ has non-degenerate products in degree $n$ if and only if $A$ admits a non-degenerate associative graded bilinear form to an abelian group concentrated in degree $n$.
\end{proposition}

\begin{proof}
Giving an associative graded bilinear form $A \times A \to X$ is equivalent to giving a graded map $A \otimes_A A \to X$ of abelian groups (where the tensor product $A\otimes_AA$ is the graded tensor product). Note that $A \otimes_A A \cong A$ as graded rings, and thus we obtain a map $A^n \to X$ of abelian groups. The claim now follows.
\end{proof}

\begin{proposition}
Let $A$ be a $\mathbb Z$-graded $k$-algebra. Then the following are equivalent
\begin{enumerate}
\item $A$ is $n$-shifted selfdual, and
\item all graded pieces $A^i$ are finite dimensional, and $A$ admits a non-degenerate associative graded bilinear form $A \times A \to k[-n]$.
\end{enumerate}
\end{proposition}

\begin{proof}
(1) $\Rightarrow$ (2): Fix an isomorphism $\Phi \colon A[n] \to \D A$ of graded left $A$-modules. Then $\Phi$ consists of isomorphisms $\Phi^i \colon A^{n+i} \to \Hom_k(A^{-i}, k)$ for all $i$. In particular we obtain isomorphisms
\[ \Hom_k((\Phi^{-n-i})^{-1}, k) \circ \Phi^i \colon A^{n+i} \to \Hom_k(\Hom_k(A^{n+i}, k), k). \]
It follows that $A^{n+i}$ is a finite dimensional $k$-vector space for all $i$.

Now we can define a non-degenerate associative graded bilinear form by
\begin{align*}
A \times A \xrightarrow{1 \times \Phi[-n]} A \times \D A[-n] & \xrightarrow{\text{ev.}} k[-n] \\
(a, \varphi) & \mapsto \varphi(a)
\end{align*}

(2) $\Rightarrow$ (1): Since the graded pieces are now assumed to be finite dimensional, having a non-degenerate bilinear form $A^i \times A^{n-i} \to k$ gives an isomorphism $A^i \to \Hom_k(A^{n-i}, k)$. Together these isomorphisms form an isomorphism $A \to \D (A[n]) = \D A[-n]$ of graded $k$-vector spaces.

The associativity of the bilinear form implies that this isomorphism respects the $A$-module structure.
\end{proof}

\begin{remark}
We note that the two propositions above in particular show that any algebra being $n$-shifted selfdual also has non-degenerate products in degree $n$.
\end{remark}

\subsection*{Regular elements}

Let $A$ be a $\mathbb Z$-graded ring. A homogeneous central element $r\in A$ is called a \emph{regular}
element of $A$ if the implication
\[ ra = 0 \quad \Longrightarrow \quad a = 0 \]
holds for all $a \in A$.

Throughout this subsection, we assume that $r$ is a central element of $A$ of positive degree, and a regular element of $A^{\geq 0}$. We assume additionally that $A$ has non-degenerate products in degree $n$.

We consider the following two subbimodules of $A$:
\begin{align*}
\Tor_r A & = \{a \in A \mid r^i a = 0 \text{ for some $i\ge 0$}\} && \text{the \emph{$r$-torsion part of $A$}} \\
I & = ( A^{\leq n} ) && \text{the ideal of $A$ generated by elements} \\
&&& \text{of degree at most $n$}
\end{align*}

\begin{theorem} \label{thm.depth1}
Let $A$ be a $\mathbb Z$-graded ring which has non-degenerate products in degree $n$. If $r$ is a homogeneous
central element of $A$ of positive degree, and a regular element of $A^{\ge 0}$, then we have the following:
\begin{enumerate}
\item $I \cdot \Tor_r A = 0 = \Tor_r A \cdot I$;
\item either $n < 0$, or $r$ acts regularly on $A$.
\end{enumerate}
\end{theorem}

\begin{proof}
For (1) it is sufficient to show that $A^{\leq n} \cdot \Tor_r A = 0 = \Tor_r A \cdot A^{\leq n}$. We show the first equality, the second one can be shown similarly.

Let $a \in A^i$ for some $i \leq n$, and $b \in (\Tor_r A)^j$. Assume $ab \neq 0$. Then there exists $c \in A^{n-i-j}$ such that $abc \neq 0$. Since $\Tor_r A$ is a subbimodule of $A$ we have
\[ bc \in (\Tor_r A)^{j + (n-i-j)} = (\Tor_r A)^{n-i} \subseteq (\Tor_r A)^{\geq 0} = 0, \]
where the last equality follows from the assumption that $r$ acts regularly on $A^{\geq 0}$. This is clearly a contradiction, so $ab = 0$.

(2): Note that if $n \geq 0$ then $1_A \in I$, so $I = A$. It follows from (1) that $\Tor_r A = 0$, so $r$ acts regularly on $A$.
\end{proof}

We observe that if $A$ is an algebra over a field, which is $n$-shifted selfdual, we obtain the following more explicit description of $I$, and of what it means that $r$ acts regularly on $A$.

\begin{lemma} \label{lem.I_Tor_dual}
Assume $A$ is a $\mathbb Z$-graded algebra over a field $k$, which is $n$-shifted selfdual, and $r$ is a central element of $A$ of positive degree and regular on $A^{\ge 0}$. Then
\[ \D I \cong (A / \Tor_r A)[n] \]
as right and as left $A$-modules.
\end{lemma}

\begin{proof}
By definition $I$ is the smallest subbimodule of $A$ containing $A^{\leq n}$. Applying $\D - [-n]$ we obtain that $\D I[-n]$ is the smallest factor bimodule of $A$, in which no elements of $A^{\geq 0}$ become zero. It is clear that this is a factor bimodule of $A / \Tor_r A$. On the other hand $r$ acts regularly on $A / \Tor_r A$, so any non-zero element $a$ of $A / \Tor_r A$ has a non-zero multiple of the form $r^i a \in A^{\geq 0}$, and thus cannot be mapped to zero.  It follows that $\D I[-n] \cong A / \Tor_r A$.
\end{proof}

\begin{lemma} \label{lem.regular}
Assume $A$ is an algebra over a field $k$, which is $n$-shifted selfdual. If $r$ is a central element of $A$ of positive degree which acts regularly on $A$, then $A$ is periodic.
\end{lemma}

\begin{proof}
It follows from duality that $r$ acting regularly is equivalent to every element being divisible by $r$. Thus $r$ acts bijectively on $A$, and thus induces isomorphisms $A^i \to A^{i + |r|}$ for all $i$.
\end{proof}

It follows from Theorem~\ref{thm.depth1} and Lemma~\ref{lem.regular} that the case $n \geq 0$ is very special. In practice it seems the most typical case is $n=-1$, which we therefore record here explicitly.

\begin{theorem}
Assume $A$ is an algebra over a field $k$, which is $(-1)$-shifted selfdual, and $r$ is a central element of
$A$ of positive degree which is regular on $A^{\ge 0}$. Then $A$ has a filtration of subbimodules
\[ 0 \subseteq \Tor_r A \subseteq I \subseteq A, \]
such that
\begin{enumerate}
\item $\Tor_r A$ is concentrated in negative degrees, $A / I$ is concentrated in non-negative degrees, and these two are dual as left and as right $A$-modules.
\item $I / \Tor_r A$ is periodic.
\end{enumerate}
\end{theorem}

\begin{proof}
We first note that $\Tor_r A$ is concentraded in negative degrees by assumption, and $I$ contains all elements of negative degree by definition. Thus we have the filtration of the theorem, and the first two statements of (1).

For the duality statement in (1) note that
\[ \D (A/I) = \Ker [ \D A \twoheadrightarrow \D I ] = (\Tor_r A)[-1] \]
by Lemma~\ref{lem.I_Tor_dual}.

For (2) note that $r$ acts regularly on $I / \Tor_r A$, and that $I / \Tor_r A$ is selfdual. The proof now follows the proof of Lemma~\ref{lem.regular}.
\end{proof}

\subsection*{Depth $\geq 2$}

In this subsection we assume that the non-negative part $A^{\geq 0}$ of a $\mathbb Z$-graded ring $A$ admits a regular sequence $(r, \tilde r)$ of length $2$ of central elements of positive degree. That means that $r$ acts regularly on $A^{\geq 0}$, and $\tilde r$ acts regularly on $A^{\geq 0} / r A^{\geq 0}$.

In this situation we obtain the following:

\begin{lemma} \label{lem.tor_depth_2}
Let $A$ be a $\mathbb Z$-graded ring, such that $A^{\geq 0}$ admits a regular sequence $(r, \tilde r)$
of central elements of $A$ of positive degree. Then
\[ \Tor_r A = A^{< 0}. \]
\end{lemma}

\begin{proof}
Assume there is $a \in A^{< 0} \setminus \Tor_r A$. Then for any $i$ we have $r^i a \neq 0$. Choose $i$ minimal such that $r^i a \in A^{\geq 0}$. Clearly $r^i a \not\in r A^{\geq 0}$, but for $j \gg 0$ we have ${\tilde r}^j r^i a = r {\tilde r}^j r^{i-1} a  \in r A^{\geq 0}$. This contradicts the assumption that $\tilde r$ acts regularly on $A^{\geq 0} / r A^{\geq 0}$.
\end{proof}

\begin{remark}
We know that $\Tor_r A$ is an $A$-subbimodule of $A$. Lemma~\ref{lem.tor_depth_2} shows that it is also a quotient $A^{\geq 0}$-bimodule of $A$, and thus
\[
A \cong A^{\geq 0} \oplus A^{<0}
\]
as $A^{\geq 0}$-bimodules.
\end{remark}

We obtain the following result immediately from Theorem~\ref{thm.depth1} and Lemma~\ref{lem.tor_depth_2}.

\begin{theorem} \label{thm.depth2}
Let $A$ be a $\mathbb Z$-graded ring having non-degenerate products in degree $n$, such that $A^{\ge 0}$ admits a regular sequence $(r, \tilde r)$ of central elements of $A$ of positive degree. Then
$n < 0$, and $A^{\leq n} \subseteq A^{< 0}$ are ideals of $A$ annihilating each other.
\end{theorem}

\begin{proof}
By Lemma~\ref{lem.tor_depth_2} we know that $A^{< 0} = \Tor_r A$ is an ideal of $A$. It follows from Theorem~\ref{thm.depth1}(2) that $n < 0$. (Note that $\Tor_r A = A^{< 0} \ne 0$, since otherwise the non-degenerate products assumption would imply $A^{> n} = 0$, clearly contradicting the existence of a regular element.)

As before we denote by $I$ the ideal generated by $A^{\leq n}$. By Theorem~\ref{thm.depth1}(1) we know that $I$ and $A^{< 0}$ annihilate each other.

It remains to show that $I = A^{\leq n}$. Assume $a \in I^i \setminus \{0\}$ for some $i > n$. By assumption there is $b \in A^{n-i} \subseteq A^{< 0}$ such that $ab \neq 0$, and this contradicts the fact that $I \cdot A^{< 0} = 0$.
\end{proof}

If we restrict to the special case that $A$ is $(-1)$-shifted selfdual we otain the following more explicite description of $A$.

\begin{corollary}\label{cor.depth2}
Let $A$ be a $\mathbb Z$-graded ring having non-degenerate products in degree $-1$. If $A^{\ge 0}$ admits a regular sequence $(r, \tilde r)$ of central elements of $A$ of positive degree, then
\[
A \cong A^{\geq 0} \ltimes A^{< 0}.
\]
If $A$ is moreover a $k$-algebra which is $(-1)$-shifted selfdual, then
$\D A^{< 0} \cong A^{\geq 0}[-1]$ both as left and as right $A^{\geq 0}$-modules.
\end{corollary}

\begin{proof}
The first statement holds since, by Theorem~\ref{thm.depth2} $A^{< 0}$ is an ideal which squares to zero. The second claim is a restatement of Lemma~\ref{lem.I_Tor_dual}, using $\Tor_r A = A^{< 0}$ (Lemma~\ref{lem.tor_depth_2}) and $I = A^{< 0}$ (proof of Theorem~\ref{thm.depth2}).
\end{proof}

\section{Calabi-Yau Categories}\label{B}

In this section we apply the general results from Section \ref{A} to endomorphism rings in Calabi-Yau triangulated categories. We fix a field $k$, and start with a general lemma, which guarantees that a certain family of bilinear forms is associative and non-degenerate. Recall first that a category $\T$ is a \emph{$k$-category} if for all objects $X,Y,Z \in \T$, the set $\Hom_{\T}(X,Y)$ is a $k$-vector space, and composition
$$
\Hom_{\T}(Y,Z) \times \Hom_{\T}(X,Y) \to \Hom_{\T}(X,Z)
$$
is $k$-bilinear.

\begin{lemma}\label{forms}
Let $k$ be a field, $\mathcal{T}$ a $k$-category, and $X$ and $Z$ objects in $\mathcal{T}$ such that there is a natural isomorphism
\[
\eta \colon \Hom_{\mathcal{T}}(X,-) \to  \D\Hom_{\mathcal{T}}(-, Z).
\]
Then the bilinear forms
\begin{align*}
\langle - , - \rangle_Y \colon \Hom_{\mathcal{T}}(Y, Z) &\otimes_k \Hom_{\mathcal{T}}(X, Y)  \to k \\
f &\otimes g  \mapsto \eta_Y(g)(f)
\end{align*}
are non-degenerate. Moreover, they are associative in the sense that $\langle f, g \circ h \rangle_Y = \langle f \circ g, h \rangle_{Y'}$ for all objects $Y' \in \T$ and all morphisms $Y \xrightarrow{f} Z$, $Y' \xrightarrow{g} Y$ and $X \xrightarrow{h} Y'$ in $\T$.
\end{lemma}

\begin{proof}
The bilinear forms are non-degenerate by definition, so we only have to check associativity. The isomorphism $\eta$ being natural means that for any morphism $Y' \xrightarrow{g} Y$ there is an equality
\[ \D \Hom_{\mathcal{T}}(g, Z) \circ \eta_{Y} = \eta_{Y'} \circ \Hom_{\mathcal{T}}(X, g). \]
Applying this to a morphism $X \xrightarrow{h} Y'$ we obtain
\[ \eta_{Y'}(h) ( - \circ g) = \eta_{Y}(g \circ h)(-), \]
and inserting a morphism $Y \xrightarrow{f} Z$ we obtain the formula of the lemma.
\end{proof}

Having established this elementary lemma, we turn now to Calabi-Yau triangulated categories. Let $\T$ be a $\Hom$-finite (i.e.\ $\Hom_{\T}(X,Y)$ is a finite dimensional $k$-vector space for all objects $X,Y$) triangulated $k$-category with suspension functor $\s \colon \T \to \T$. A \emph{Serre functor} on $\T$ is an equivalence $S \colon \T \to \T$ of triangulated $k$-categories, together with functorial isomorphisms $\Hom_{\T}(X,Y) \cong \D \Hom_{\T}(Y,SX)$ of vector spaces for all objects $X,Y \in \T$. By \cite{BondalKapranov}, such a functor is unique if it exists. For an integer $n \in \mathbb{Z}$, the category $\T$ is called \emph{$n$-Calabi-Yau} if it admits a Serre functor which is isomorphic as a triangle functor to $\s^n$ (cf.\ \cite{Keller}). In this case there are functorial isomorphisms $\Hom_{\T}(X,Y) \cong \D \Hom_{\T}(Y,\s^nX)$ of vector spaces for all objects $X,Y \in \T$.

We show next that Lemma \ref{forms} implies these functorial isomorphisms induce a non-degenerate associative graded bilinear form on the graded endomorphism ring of an object in $\T$.  We first recall briefly some facts about such
graded endomorphism rings. For an object $X\in\T$, the graded endomorphism ring of $X$ is the graded vector space
\[
A=\oplus_{i\in\mathbb Z}A^i \quad\text{with}\quad A^i=\Hom_{\T}(X,\s^iX),
\]together with a product defined as follows.  For $f\in A^i$ and $g\in A^j$, we identify $f$ with its image under the isomorphism $A^i\cong\Hom_{\T}(\s^{j}X,\s^{i+j}X)$ and define $fg\in A^{i+j}$ by composition.

\begin{proposition} Let $k$ be a field, $\T$ an $n$-Calabi-Yau triangulated $k$-category, and $X$ an object in $\T$. Let  $A=\oplus_iA^i$ denote the graded endomorphism ring of $X$, namely, the
$\mathbb{Z}$-graded $k$-algebra with $A^i=\Hom_{\T}(X,\s^iX)$.
Then there exists a non-degenerate associative graded bilinear form
\[
\langle-,-\rangle:A\otimes_kA\to k[-n].
\]
\end{proposition}

\begin{proof}
Since $\T$ is $n$-Calabi-Yau, there are functorial isomorphisms
\[
\Hom_{\T}(X,\s^{n-i}X)\cong \D\Hom_{\T}(\s^{n-i}X,\s^nX)
\]
for all $i$.  Therefore from Lemma \ref{forms} we have a family of non-degenerate bilinear forms
\[
\langle-,-\rangle^i:\Hom_{\T}(\s^{n-i}X,\s^nX)\otimes_k\Hom_{\T}(X,\s^{n-i}X) \to k.
\]
Moreover, this family of bilinear forms is associative in the sense that for $\s^{n-i}X \xrightarrow{f} \s^nX$, $\s^{n-j}X \xrightarrow{g} \s^{n-i}X$ and $X \xrightarrow{h} \s^{n-j}$ in $\T$, we have
$\langle f, g \circ h \rangle^i = \langle f \circ g, h \rangle^j$. Using the functorial isomorphisms $A^i\cong\Hom_{\T}(\s^{n-i}X,\s^{n}X)$ for all $i$, we thus obtain a non-degenerate associative graded bilinear form
\[
A\otimes_k A\to k[-n],
\]
\end{proof}

We may therefore apply the results of Section \ref{A} to the graded endomorphism rings of objects in $n$-Calabi-Yau categories.  For future reference we summarize these in the following.

As in the previous section we let $\Tor_r A  = \{a \in A \mid r^i a = 0 \text{ for some $i\ge 0$}\}$ denote the $r$-torsion part of $A$, and $I  = ( A^{\leq n} )$ denote the ideal of $A$ generated by elements of degree at most $n$.

\begin{theorem}\label{thm:calabi-yau}
Let $k$ be a field, and $\T$ an $n$-Calabi-Yau triangulated $k$-category. Furthermore, let $X$ be an object in $\T$, and denote by $A$ the $\mathbb{Z}$-graded $k$-algebra with $A^i=\Hom_{\T}(X,\s^iX)$.

If $r$ is a homogeneous central element of $A$ of positive degree which is regular on $A^{\ge 0}$, then
\begin{enumerate}
\item $I\cdot \Tor_rA=0=\Tor_rA\cdot I$;
\item either $n<0$, or $r$ acts regularly on $A$, and $A$ is periodic;
\item $\D I\cong(A/\Tor_rA)[n]$ as right and as left $A$-modules.
\end{enumerate}

If in addition $A^{\ge 0}$ admits a regular sequence $(r, \tilde r)$ of central elements of $A$ of positive degree, then
\begin{enumerate}
\item[(4)] $\Tor_rA=A^{<0}$;
\item[(5)] $A\cong A^{\ge 0} \ltimes A^{<0}$;
\item[(6)] $A^{\le n}\subseteq A^{<0}$ are ideals of $A$ annihilating one another.
\end{enumerate}
\end{theorem}\qed

We can further apply Theorem \ref{thm:calabi-yau} and the results of Section \ref{A} to Tate cohomology rings of objects in the stable module category of a symmetric algebra. Recall that a finite dimensional $k$-algebra $\La$ is \emph{symmetric} if there exists an isomorphism $\La \cong \D( \La )$ of $\La$-bimodules. Such an algebra is in particular selfinjective, and its stable module category $\stmod \La$ (we consider only finitely generated left modules) is a $\Hom$-finite triangulated $k$-category with suspension functor $\s = \Omega_{\La}^{-1}$. Given an integer $i \in \mathbb{Z}$ and two $\La$-modules $M,N$, the $i$th \emph{Tate cohomology} group $\Tateext_{\La}^i(M,N)$ is defined as
$$\Tateext_{\La}^i(M,N) \stackrel{\text{def}}{=} \stHom_{\La}(M, \Omega_{\La}^{-i}(N)),$$
that is, the vector space $\Hom_{\stmod \La}(M, \s^i N)$. When $M=N$, the graded $k$-vector space $\Tateext_{\La}^*(M,M) = \oplus_{i \in \mathbb{Z}} \Tateext_{\La}^i(M,M)$ becomes a ring, the Tate cohomology ring of $M$.

\begin{proposition} Let $k$ be a field, $\La$ a finite dimensional symmetric $k$-algebra. Then
the stable module category $\stmod \La$ is $(-1)$-Calabi-Yau.
\end{proposition}

\begin{proof}
Given finitely generated $\La$-modules $M$ and $N$, there is a functorial isomorphism
$$\stHom_{\La}(M,N) \cong \D \stHom_{\La}(N, \tau \Omega_{\La}^{-1}(M)),$$
where $\tau$ is the Auslander-Reiten translate (cf.\ \cite{KrauseLe} or \cite{Buchweitz}). Also, by the result \cite[Proposition IV.3.8]{AuslanderReitenSmalo}, the Auslander-Reiten translate is naturally isomorphic to $\Omega_{\La}^{2}$, hence we obtain a functorial isomorphism
$$\stHom_{\La}(M,N) \cong \D \stHom_{\La}(N, \Omega_{\La}^{1}(M)).$$
Consequently, the stable module category $\stmod \La$ is $(-1)$-Calabi-Yau.
\end{proof}

\begin{theorem}\label{thm:symmetric}
Let $k$ be a field, $\La$ a finite dimensional symmetric $k$-algebra and $M$ a finitely generated left module. Denote by $A$ the $\mathbb{Z}$-graded $k$-algebra with $A^i= \Tateext_{\La}^i(M,M)$.

If $r$ is a homogeneous central element of $A$ of positive degree which is regular on $A^{\ge 0}$, then
\begin{enumerate}
\item $I\cdot \Tor_rA=0=\Tor_rA\cdot I$;
\item $A$ has a filtration of subbimodules $0\subseteq \Tor_rA\subseteq I \subseteq A$
such that
\begin{itemize}
\item $\Tor_rA$ is concentrated in negative degrees, $A/I$ is concentrated in non-negative degrees, and
these two are dual as left and as right $A$-modules;
\item $I/\Tor_rA$ is periodic.
\end{itemize}
\end{enumerate}

If in addition $A^{\ge 0}$ admits a regular sequence $(r, \tilde r)$ of central elements of $A$ of positive degree, then
\begin{enumerate}
\item[(3)] $\Tor_rA=A^{<0}$;
\item[(4)] $(A^{<0})^2=0$ and $A\cong A^{\ge 0} \ltimes A^{<0}$;
\item[(5)] $\D A^{<0}\cong A^{\ge 0}[-1]$ both as left and as right $A^{\ge 0}$-modules.
\end{enumerate}
\end{theorem}\qed

The group algebra of a finite group is symmetric, and so Theorem \ref{thm:symmetric} applies to Tate cohomology rings of modules over such an algebra. In particular, the theorem applies to the Tate group cohomology ring, which is just the Tate cohomology ring of the trivial module. We thus recover \cite[Theorem 3.1]{BensonCarlson}. Note that the Tate group cohomology ring is always graded commutative, hence homogeneous elements in odd degree square to zero when the characteristic of the underlying field is not two. Therefore, when the characteristic of the underlying field is not two, a homogeneous regular sequence must consist of elements in even degrees. Consequently, in any characteristic, a homogeneous regular sequence must automatically be central.

\begin{corollary}\label{cor:groups}
Let $k$ be a field, $G$ a finite group whose order is divisible by the characteristic of $k$, and denote by $A$ the Tate group cohomology ring $\Tateext_{kG}^*(k,k)$. Then the conclusions of Theorem \ref{thm:symmetric} hold.
\end{corollary}

As a second application of Theorem \ref{thm:symmetric}, we consider the Tate analogue of Hochschild cohomology rings. Let $\La$ be a finite dimensional symmetric $k$-algebra. Then by \cite[Corollary 3.3]{BerghJorgensen}, its enveloping algebra $\Lae = \La \otimes_k \La^{\op}$ is also symmetric. Given an integer $i \in \mathbb{Z}$, the $i$th \emph{Tate-Hochschild cohomology group} $\TateHH^i( \La )$ is defined as
$$\TateHH^i( \La ) \stackrel{\text{def}}{=} \Tateext_{\Lae}^i( \La, \La ),$$
and the graded ring $\TateHH^*( \La )$ is the \emph{Tate-Hochschild cohomology ring} of $\La$. In positive degree, this ring coincides with the ordinary Hochschild cohomology ring, i.e.\ $\TateHH^{\ge 1}( \La ) = \HH^{\ge 1} ( \La )$. As with group cohomology rings of finite groups, Tate-Hochschild cohomology rings are graded commutative, hence homogeneous regular sequences must be central.

\begin{corollary}\label{cor:TateHH}
Let $k$ be a field, $\La$ a finite dimensional symmetric $k$-algebra, and denote by $A$ the Tate-Hochschild cohomology ring $\TateHH^*( \La )$.  Then the conclusions of Theorem \ref{thm:symmetric} hold.
\end{corollary}

We include an example illustrating Corollary \ref{cor:TateHH}.

\begin{example}\label{ex:TateHHci}
Let $c$ be a positive integer and $\La$ the commutative local complete intersection
$$k[x_1, \dots, x_c]/(x_1^{a_1}, \dots, x_c^{a_c}),$$
where $a_1, \dots, a_c$ are integers, all at least $2$. This is a finite dimensional symmetric $k$-algebra. We denote the characteristic of the ground field $k$ by $p$.

Suppose that $c=1$, and denote $x_1$ and $a_1$ by just $x$ and $a$, so that $\La = k[x]/(x^a)$. Then it follows from \cite[Theorem 4.4]{BerghJorgensen} that
$$\dim \TateHH^n( \La ) = \left \{
\begin{array}{ll}
a-1 & \text{ if $p \nmid a$} \\
a & \text{ if $p \mid a$}
\end{array} \right. $$
for every integer $n \in \mathbb{Z}$. It is well known that $\La$ is periodic as a bimodule over itself, with period at most $2$. Also, by  \cite[Theorem 7.1]{Holm}, its ordinary Hochschild cohomology ring $\HH^* ( \La )$ is given by
$$\HH^* ( \La ) \cong \left \{
\begin{array}{ll}
k[u,v,w]/(u^a,wu^{a-1},vu^{a-1},v^2) & \text{ if $p \nmid a$} \\
k[u,v,w]/(u^a,v^2) & \text{ if $p \mid a$ and $p \neq 2$} \\
k[u,v,w]/(u^a,v^2-wu^{a-2}) & \text{ if $p \mid a, p=2$ and $a \equiv 2 ( \mod 4 )$} \\
k[u,v,w]/(u^a,v^2) & \text{ if $p \mid a, p=2$ and $a \equiv 0 ( \mod 4 )$}
\end{array} \right. $$
where $|u|=0,|v|=1$ and $|w|=2$. The element $w$ is the ``periodicity element", and the element $u$ corresponds to the element $x \in \La$. When passing to the Tate-Hochschild cohomology ring $\TateHH^*( \La )$, the socle element $x^{a-1}$, i.e.\ the element $u^{a-1}$, vanishes when $p \nmid a$. Therefore the non-negative part $\TateHH^{\ge 0}( \La )$ of the Tate-Hochschild cohomology ring is given by
$$\TateHH^{\ge 0}( \La ) \cong \left \{
\begin{array}{ll}
k[u,v,w]/(u^{a-1},v^2) & \text{ if $p \nmid a$} \\
k[u,v,w]/(u^a,v^2) & \text{ if $p \mid a$ and $p \neq 2$} \\
k[u,v,w]/(u^a,v^2-wu^{a-2}) & \text{ if $p \mid a, p=2$ and $a \equiv 2 ( \mod 4 )$} \\
k[u,v,w]/(u^a,v^2) & \text{ if $p \mid a, p=2$ and $a \equiv 0 ( \mod 4 )$}.
\end{array} \right. $$
Consequently, there does not exist a regular sequence in $\TateHH^{\ge 0}( \La )$ of length $2$ of homogeneous elements of positive degree; the periodicity element $w$ forms a maximal homogeneous regular sequence. Of course, this is not unexpected; since $\La$ is periodic as a bimodule, there are lots of negative degree homogeneous elements $\eta, \theta \in \TateHH^*( \La )$ with $\eta \theta \neq 0$.

Next, suppose that $c \ge 2$. Since $\La$ is isomorphic to the tensor product
$$k[x_1]/(x_1^{a_1}) \otimes_k \cdots \otimes_k k[x_c]/(x_c^{a_c}),$$
its Hochschild cohomology ring $\HH^* ( \La )$ is isomorphic to a twisted tensor product
$$\HH^* (k[x_1]/(x_1^{a_1})) \widehat{\otimes}_k \cdots \widehat{\otimes}_k \HH^* (k[x_c]/(x_c^{a_c})).$$
This is almost the same as the usual tensor product, the only difference is that elements of odd degrees anticommute. From the proof of \cite[Theorem 4.4]{BerghJorgensen} it follows that the dimension of $\TateHH^0( \La )$ is given by
$$\dim \TateHH^0( \La ) = \left \{
\begin{array}{ll}
\prod_{i=1}^c a_i -1 & \text{ if $p \nmid a_i$ for all $i$} \\
\prod_{i=1}^c a_i & \text{ if $p \mid a_i$ for one $i$}.
\end{array} \right. $$
\sloppy Consequently, when $p$ divides one of the $a_i$, then when passing from $\HH^* ( \La )$ to $\TateHH^*( \La )$, the socle element $x_1^{a_1} \cdots x_c^{a_c}$ vanishes. At any rate, regardless of the characteristic $p$, we see that $\TateHH^{\ge 0}( \La )$ contains a regular sequence of length $c$ of homogeneous elements of positive degrees. For instance, we could take the sequence $w_1, \dots, w_c$, where each $w_i$ corresponds to the periodicity element in $\TateHH^* (k[x_i]/(x_i^{a_i}))$. By Corollary \ref{cor:TateHH}, if $\eta, \theta \in \TateHH^*( \La )$ are negative degree homogeneous elements, then $\eta \theta = 0$.
\end{example}

\section{Commutative local rings}\label{C}

In this final section we consider negative products in Tate cohomology over commutative local (meaning also Noetherian) rings. Let $R$ be such a ring, and denote the functor $\Hom_R(-,R)$ by $(-)^*$. Recall that a finitely generated $R$-module $X$ is called \emph{totally reflexive} if the following three conditions hold:
\begin{enumerate}
\item $X$ is reflexive, that is, the canonical homomorphism $X \to X^{**}$ is an isomorphism.
\item $\Ext_R^i(X,R)=0$ for $i \ge 1$.
\item $\Ext_R^i(X^*,R)=0$ for $i \ge 1$.
\end{enumerate}
Such a module is also called a module of \emph{Gorenstein dimension zero}. Now recall that a complex
$$T: \quad \cdots \to F_2 \xrightarrow{\partial_2^T} F_1 \xrightarrow{\partial_1^T} F_0 \xrightarrow{\partial_0^T} F_{-1} \to F_{-2} \to \cdots$$
of $R$-modules is called a \emph{complete resolution} if the modules $F_i$ are finitely generated free and both $T$ and $T^*$ are acyclic. It is well known that a module $X$ is totally reflexive if and only if it is the image $\Im \partial_0^T$ of a map in such a complete resolution $T$. In this case, given any $R$-module $Y$, the $i$th \emph{Tate cohomology} group $\Tateext_R^i(X,Y)$ is defined as
$$\Tateext_{R}^i(X,Y) \stackrel{\text{def}}{=} \H^i \left ( \Hom_R(T,Y) \right ).$$
That is, the group $\Tateext_{\La}^i(X,Y)$ is the $i$th cohomology of the complex $\Hom_R(T,Y)$. When $Y=X$, then the graded $R$-module $\Tateext_{R}^*(X,X) = \oplus_{i \in \mathbb{Z}} \Tateext_{R}^i(X,X)$ becomes an $R$-algebra called the \emph{Tate cohomology ring of $X$}, where the multiplication is given by composition of chain maps on the complete resolution of $X$.

It should be noted that Tate cohomology can be defined more generally when $X$ has \emph{finite Gorenstein dimension}, not only when $X$ is totally reflexive. However, we shall only consider modules that are totally reflexive. Note also that the ring $R$ is Gorenstein if and only every maximal Cohen-Macaulay $R$-module is totally reflexive \cite{AuslanderBridger}. In particular, if $R$ is a zero-dimensional Gorenstein ring (i.e.\ if $R$ is selfinjective), then every $R$-module is totally reflexive.

\subsection*{Commutative local Gorenstein rings.}
Let $R$ be a commutative local Gorenstein ring with maximal ideal $\m$ and Krull dimension $d$. As mentioned above, in this case all maximal Cohen-Macaulay $R$-modules are totally reflexive.  Recall that a finitely generated  $R$-module $X$ is said to be \emph{free on the punctured spectrum of $R$} if $X_p$ is a free $R_p$-module for all non-maximal primes ideals $p$ of $R$. By \cite[7.7.5(i)]{Buchweitz}, when $X$ is free on the punctured spectrum of $R$ we have for all $i$ natural isomorphisms
\[
\eta^i:\Tateext_R^i(X,-)\to\Hom_R(\Tateext_R^{d-1-i}(-,X),E)
\]
where $E$ denotes the injective hull of the residue field $R/\m$ of $R$.
It follows that we have a graded bilinear form
\begin{equation}\label{nagbf}
\langle-,-\rangle:A\otimes_RA \to E[1-d]
\end{equation}
where $A=\oplus_{i\in\mathbb Z}\Tateext_R^i(X,X)$, and
$\langle-,-\rangle^i: A^i\otimes_RA^{d-1-i}\to E[1-d]$ is defined by
$\langle f,g\rangle^i=\eta^{d-1-i}(g)(f)$.

The following result is a special case of \cite[7.7.5(iii)]{Buchweitz}.

\begin{proposition}\label{Gor:nagbf}  Let $R$ be a commutative local Gorenstein ring of dimension $d$, and $X$ be a maximal  Cohen-Macualay $R$-module which is free on the punctured spectrum of $R$.  Then the graded bilinear form (\ref{nagbf}) is non-degenerate and associative.
\end{proposition}

We now make statements in terms of regular sequences as in the previous section.  However, we encounter a dichotomy due to (2) of Theorem \ref{thm.depth2}, since the index of duality $n=d-1$ may be non-negative.  We treat the $d=0$ and $d>0$ cases separately, starting with the former.  In this case all $R$-modules are free on the punctured spectrum vacuously.

\begin{theorem}\label{Gorenstein}
Let $(R, \m)$ be a zero-dimensional commutative local Gorenstein ring, and denote by $A$ the Tate cohomology ring $\Tateext_{R}^*(X,X)$ of the finitely generated $R$-module $X$. If $A^{\ge 0}$ contains a regular sequence of length two of central elements of positive degree, then
\[
A\cong A^{\ge 0}\ltimes A^{<0}.
\]
\end{theorem}

\begin{remark}
In \cite{AvramovVeliche} additional structure is specified in the special case of $M=R/\m$, but for any dimension.
\end{remark}

\begin{theorem}
Let $(R, \m)$ be a commutative local Gorenstein ring of dimension $d>0$, and denote by $A$ the Tate cohomology ring $\Tateext_{R}^*(X,X)$ of the finitely generated maximal Cohen-Macaulay $R$-module $X$ which is free on the punctured spectrum of $R$. Then  $A^{\ge 0}$ contains at most a regular sequence of length one of central homogeneous elements of positive degree.
\end{theorem}

\begin{proof}
The result follows from Proposition \ref{Gor:nagbf} and (3) of Theorem \ref{thm.depth2}.
\end{proof}

Since $\Ext_R^i(X,X)\cong\Tateext_R^i(X,X)$ for all $i>0$, and $\Tateext^0_R(X,X)$ is a quotient
of $\Ext^0_R(X,X)$ when $X$ is a maximal Cohen-Macaulay $R$-module, we get the following
corollary involving the absolute Ext algebra.

\begin{corollary}
Let $(R, \m)$ be a commutative local Gorenstein ring of Krull dimension $d>0$, and $X$ a finitely generated maximal Cohen-Macaulay $R$-module which is free on the punctured spectrum of $R$.  Then
$\Ext_R^*(X,X)$ contains at most a regular sequence of length one of central homogeneous elements of positive degree.\qed
\end{corollary}

When $(R, \m)$ is a commutative local Gorenstein ring of positive Krull dimension, then the residue field $k=R/\m$ is never a maximal Cohen-Macaulay $R$-module.  In contrast to the above corollary, Lemma 8.3 of \cite{AvramovVeliche} shows that if $R$ is moreover a complete intersection, then the absolute Ext algebra, $\Ext_R^*(k,k)$, does have a regular sequence of central elements of length equal to the codimension of $R$. Hence when the codimension is at least two (i.e.\ when $R$ is not a hypersurface) then the product of two negative degree elements in $\Tateext_{R}^*(k,k)$ is zero. We illustrate this (and Theorem \ref{Gorenstein}) with an example, using the same ring as in the example in Section \ref{B}.

\begin{example}
Let $c$ be a positive integer and $R$ the commutative local zero-dimensional complete intersection
$$k[x_1, \dots, x_c]/(x_1^{a_1}, \dots, x_c^{a_c}),$$
where $a_1, \dots, a_c$ are integers, all at least $2$. It follows from \cite[Theorem 5.3]{BerghOppermann} that the $\Ext$-algebra of $k$ is given by
$$\Ext_R^*(k,k) = k \langle y_1, \dots, y_c, z_1, \dots, z_c \rangle /I,$$
where $I$ is the ideal in $k \langle y_1, \dots, y_c, z_1, \dots, z_c \rangle$ generated by the following relations:
$$\begin{array}{ll}
z_iz_j - z_jz_i \\
z_iy_j - y_jz_i \\
y_iy_j + y_jy_i & \text{for $i \neq j$} \\
y_i^2 - z_i & \text{if $a_i=2$} \\
y_i^2  & \text{if $a_i \neq 2$.}
\end{array}$$
The degrees of the homogeneous generators are given by $|y_i| =1$ and $|z_i| =2$. For example, when all the exponents $a_i$ are at least $3$, then
$$\Ext_R^*(k,k) = \wedge^* (y_1, \dots, y_c) \otimes_k k[z_1, \dots, z_c],$$
where $\wedge^* (y_1, \dots, y_c)$ is the exterior algebra on the generators $y_1, \dots, y_c$.

The depth of $\Ext_R^*(k,k)$ is $c$, and the sequence $z_1, \dots, z_c$ is a maximal regular sequence in $\Ext_R^*(k,k)$. When $c=1$, then $k$ is a periodic module, and so $\Tateext_{R}^*(k,k)$ contains lots of nonzero products of elements of negative degrees. When $c \ge 2$, then Theorem \ref{Gorenstein} shows that the product of two negative degree elements in $\Tateext_{R}^*(k,k)$ is always zero.
\end{example}

\subsection*{General non-degeneracy}

Finally, we include a discussion and result which explains in a transparent manner the non-degeneracy of the graded bilinear form
\[
\langle-,-\rangle: \Tateext_R^*(R/\m,R/\m)\otimes_R \Tateext_R^*(R/\m,R/\m) \to R[1].
\]
in the case where $R$ is a zero-dimensional commutative Gorenstein ring.

Let $I$ be an ideal of the commutative local ring $(R,\m)$, and assume that $R/I$ has a complete resolution of the form
\[
T: \quad\cdots \to R^{b_2} \xrightarrow{\partial_2} R^{b_1} \xrightarrow{\partial_1} R \xrightarrow{s} R \to \cdots
\]
so that $s$ factors through the natural map $R\to R/I$. Furthermore, assume that $M$ is a totally reflexive $R$-module whose complete resolution $U$ satisfies
$\Im\partial^U_i\subseteq I U_{i-1}$ for all $i$.  We define a non-degenerate pairing
\[
\Tateext^{-i-1}_R(X,R/I) \otimes_{R} \Tateext^i_R(R/I,X) \to R/I
\]
as follows.  Choose a  nonzero element $\eta\in\Tateext^i_R(R/I,X)$. Assume that $\eta$ is represented by the top vertical map in the following diagram:
\[
\xymatrix @R+20 pt {
R \ar@{->}[r]^s & R \ar@{->}[d]^{\left[\begin{smallmatrix} a_1 \\ \vdots \\ a_n \end{smallmatrix}\right]} & \\
U_{-i} \ar@{->}[r]^{\partial^U_{-i}} & U_{-i-1}  \ar@{->}^{e_j}[d] & \\
&  R \ar@{->}[r]^s & R
}
\]
Since $\eta$ is assumed to be nonzero, $a_j\notin I$ for some $j$.  Letting $e_j$ denote the unit row matrix with $1$ in the $j$th component, we see that the composition of the vertical maps is just multiplication by $a_j$.  From the assumption on the differentials of $U$, it follows that this map $e_j$ corresponds to an element $\theta\in\Tateext^{-i-1}_R(X,R/I)$.  Thus we obtain a non-degenerate pairing by setting $\langle\theta,\eta \rangle = \overline a_j\in R/I$.  A similar argument holds if we take $\eta\in\Tateext^{-i-1}_R(X,R/I)$ nonzero. Moreover, this pairing is obviously associative in the sense that, if $\eta \in\Tateext^i_R(R/I,X)$, $\phi \in\Tateext^j_R(X,X)$ and $\theta\in\Tateext^{-i-j-1}_R(X,R/I)$, then $\langle \theta \phi, \eta \rangle = \langle \theta, \phi \eta \rangle$. From Section \ref{A} we therefore obtain the following result.

\begin{theorem}\label{ideal}
Let $R$ be a commutative local ring, $I$ an ideal, and assume that $R/I$ has a complete resolution of the form
\[
\cdots \to R^{b_2} \xrightarrow{\partial_2} R^{b_1} \xrightarrow{\partial_1} R \xrightarrow{\partial_0} R \to \cdots
\]
with $\Im \partial_i \subseteq IR^{b_{i-1}}$ for all $i$. Furthermore, denote by $A$ the Tate cohomology ring $\Tateext_{R}^*(R/I,R/I)$, and assume that $A^{\ge 0}$ contains a regular sequence of length two of central homogeneous elements of positive degree. Then
\[
A\cong A^{\ge 0}\ltimes A^{<0}.
\]
\end{theorem}

\end{document}